\documentclass{article}
\usepackage{graphicx} 

\usepackage{natbib}
\usepackage{graphicx}
\usepackage{amsfonts}
\usepackage{amsmath}
\usepackage{amssymb}
\usepackage{amsthm}
\usepackage{faktor}
\usepackage{tikz-cd}
\usepackage[all]{xy}
\usepackage{enumitem}
\usepackage{tikz}
\usetikzlibrary{calc,shapes,positioning,patterns,arrows,decorations.pathreplacing}
\tikzset{Bullet/.style={fill=black,draw,color=#1,circle,minimum size=3pt,scale=0.75}}
\usepackage{xcolor}
\newcommand*{\defn}[1]{{\color{blue}\emph{#1}}}

\usepackage{mathtools}

\usepackage{appendix}
\usepackage[ddmmyyyy]{datetime}

\newtheorem{theorem}{Theorem}[section]
\newtheorem{corollary}{Corollary}[theorem]
\newtheorem{lemma}[theorem]{Lemma}
\newtheorem{proposition}[theorem]{Proposition}
\theoremstyle{definition}

\theoremstyle{remark}

\newtheorem{examplex}{Example}[section]
\newenvironment{example}
    {\pushQED{\qed}\examplex}
    {\popQED\endexamplex}

\DeclareMathOperator{\colspace}{ColSpace}
\DeclareMathOperator{\rank}{rank}

\DeclareMathOperator{\link}{Lk}

\DeclareMathOperator{\Star}{Star}

\title{Minimal Face Numbers for Volume Rigidity}
\author{Jack Southgate}
\date{\today}

\begin{document}

\maketitle

\begin{abstract}
Maxwell introduced a necessary minimum number of edges in terms of the number of vertices required for a graph to yield a Euclidean rigid generic framework in $\mathbb{R}^3$, this count was generalised to $\mathbb{R}^d$, for all $d\geq 1$. In this paper, we give the analogous minimum number of $k$-simplices, for all $0\leq k\leq d$, required for a pure $d$-dimensional simplicial complex to yield a volume rigid generic framework in $\mathbb{R}^d$, for all $d\geq 1$. In order to do so, we prove some basic facts about the volume rigidity matroid and use exterior algebraic shifting, a recently added tool to the study of volume rigidity. We later prove a volume rigidity Vertex Removal Lemma and use our count to strengthen the statement.
\end{abstract}

\section{Introduction}

Given a $d$-dimensional simplicial complex $\Sigma$, we may embed $\Sigma$ in $\mathbb{R}^d$ by listing out the positions of its vertices.
Given such an embedding, each $d$-simplex encloses a (signed) $d$-dimensional volume, which we may measure.
We say that $\Sigma$ is $d$-volume rigid if, given a \textit{typical} embedding, any continuous deformation of its vertices that preserves the $d$-volume of its $d$-simplices preserves every $d$-volume enclosed by a $(d+1)$-tuple of vertices.
Notice that there is a family of continuous deformations taking the whole simplex to its image under a $d$-volume preserving affine transformation of $\mathbb{R}^d$, we call such deformations the trivial motions or flexes of $\Sigma$.

This setup is inspired by Euclidean graph rigidity, where we embed a graph $G=(V,E)$ in $\mathbb{R}^d$ and call it Euclidean $d$-rigid if every continuous deformation of its vertices that preserves the lengths of its edges preserves the distances between all the pairs of vertices.

In Euclidean graph rigidity (as well as graph rigidity over other normed spaces), \textit{Maxwell counts} (introduced in \citet{maxwell1864calculation}) determine the minimal numbers of edges required for a graph to be $d$-rigid.
In particular, if the graph $G=(V,E)$ is Euclidean $d$-rigid, then $G$ admits a spanning subgraph $G'=(V,E')$ so that $i_{G'}(X)\leq d|X|-\binom{d+1}{2}$, for all $X\subseteq V$ with $|X|\geq d$, where $i_{G'}(X)$ denotes the number of edges in the subgraph of $G'$ induced by $X$, and $|E'|=d|X|-\binom{d+1}{2}$.

Also arising from the study of Euclidean graph rigidity is the $d$-dimensional Euclidean rigidity matroid $\mathcal{E}_n^d$, the ground set of which is the edge set of the complete graph on $n$ vertices, $K_n^1$, and the bases of which are the edge sets of subgraphs of $K_n^1$ that are \textit{minimally} Euclidean rigid in $\mathbb{R}^d$.

In going from Euclidean bar-joint rigidity to $d$-volume rigidity of pure $d$-dimensional simplicial complexes, we notice that we do not just have vertices and edges (0-  and 1-simplices) to consider, but $k$-simplices for all $0\leq k\leq d$.
It often makes sense to consider just the 0- and $d$-simplices, as we care about deformations of the 0-simplices and volumes enclosed by the $d$-simplices, and the matroid arising from the volume measurement map is on a ground set of $d$-simplices.
\citet{borcea2013realizations} deduced that if a $d$-dimensional simplicial complex on $n$ vertices is $d$-volume rigid in $\mathbb{R}^d$, then it has at least $dn-(d^2+d-1)$ maximal simplices.

In this paper we will deduce similar necessary lower bounds in terms of $n$ and $d$ for the number of $k$-simplices, for all $0\leq k\leq d$.

In Section \ref{sec:gdvr}, we formally define $d$-volume rigidity.
We then introduce the $d$-volume rigidity matroid, with the intention of finding necessary combinatorial conditions for a simplicial complex to be a basis in the matroid (these being the rigid $d$-dimensional complexes on $n$ vertices with $dn-(d^2+d-1)$ $d$-simplices).
We end by giving a combinatorial characterisation of rigidity when $d\in\{1,2\}$.

By combinatorial characterisations of the rigidity, we refer to characterisation in terms of properties of the simplicial complex that can be read off reasonably directly.
We will see in Section \ref{sec:gdvr} that the ranks of symbolic matrices also characterise the matroid, as do the ranks of the corresponding matrices with random entries (with high probability).

In Section \ref{sec:eas}, we consider exterior algebraic shifting as a tool to study $d$-volume rigidity.
\citet{bulavka2022volume} have already characterised $d$-volume rigidity in terms of the exterior algebraically shifted complexes of a simplicial complex.
We exploit their result, as well as the fact that exterior algebraic shifting preserves some helpful combinatorial properties of simplicial complexes that we noted in Section \ref{sec:gdvr}, to deduce our main result: a sharp lower bound on the face numbers of minimally rigid simplicial complexes.

Finally, in Section \ref{sec:vrl}, we take a short look at generic global volume rigidity, translating a result from Euclidean graph rigidity to volume rigidity.

Before we proceed, we will outline a few definitions and notational points:

A simplicial complex $\Sigma$ on $n$ vertices is a collection of subsets of $[n]:=\{1,\dots,n\}$ so that if $A\subseteq [n]$ is in $\Sigma$, then so are all of its subsets (including the empty set $\varnothing$).
Since elements of $\Sigma$ are collections of distinct natural numbers, we order them as tuples (noting that this does induce an orientation).
Call a $(k+1)$-tuple in $\Sigma$ a $k$-simplex and denote the set of $k$-simplices in $\Sigma$ by $\Sigma_k$.
Define the dimension of $\Sigma$ to be the largest $d$ so that $\Sigma$ contains a $d$-simplex.
If every $k$-simplex is contained in some $d$-simplex, where $d$ is the dimension of $\Sigma$, then $\Sigma$ is pure and $d$-dimensional.
If $\Sigma$ is $d$-dimensional, the \defn{$f$-vector} of $\Sigma$ is the vector $f(\Sigma)\in\mathbb{Z}^{d+2}$, with entries indexed by $\{-1,0,1,\dots,d\}$, so that
\begin{equation*}
    f(\Sigma)_k=\begin{cases}
        1\text{, }&\text{if }k=-1,\\
        |\Sigma_k|\text{, }&\text{if }k\geq 0,
    \end{cases}
\end{equation*}
ie. $f(\Sigma)$ lists out the number of $k$-dimensional faces of $\Sigma$, with $(-1)^{\text{th}}$ entry equal to 1 (by convention).
The \defn{Betti numbers} of $\Sigma$ are roughly the number of $d$-dimensional holes in $\Sigma$.
They are usually defined in terms of the homology of $\Sigma$, where the $k^{\text{th}}$ Betti number, $\beta(\Sigma)_k$, of $\Sigma$ is the dimension of the $k^{\text{th}}$ homology group of $\Sigma$. A more in depth introduction may be found in \citet{hatcher2002algebraic}.

Given a $k$-simplex $i_1\dots i_{k+1}$ in the $d$-dimensional complex $\Sigma$, the \defn{star} of $i_1\dots i_{k+1}$, denoted $\Star_{\Sigma}(i_1\dots i_{k+1})$ is the simplicial complex defined by the maximal simplices containing $i_1\dots i_{k+1}$ as a sub-simplex.
The \defn{link} of $i_1\dots i_{k+1}$, denoted $\link_{\Sigma}(i_1\dots i_{k+1})$, is the collection of $(d-k-1)$-simplices $j_1\dots j_{d-k}$ so that the simplex obtained by concatenation, $i_1\dots i_{k+1} j_1\dots j_{d-k}$ (after a potential re-ordering), is a maximal simplex in $\Star(i_1\dots i_{k+1})$.

Next, suppose that a vector space $\mathbf{k}^n$ has basis $\{e_{i_1},\dots,e_{i_n}\}$, where $I=\{i_1,\dots, i_n\}$, is a finite indexing set, we may write $\mathbf{k}^n$ as $\mathbf{k}^I$.
Let $J\subseteq I$ be a subset of indices, the \defn{orthogonal projection} of $\mathbf{k}^I$ onto $\mathbf{k}^J$ is the linear map $\pi_J:\mathbf{k}^I\rightarrow\mathbf{k}^J$ with kernel $\mathrm{Span}_{\mathbf{k}}\{e_i:i\in I\setminus J\}$.

Similarly, given a matrix $A\in\mathbf{k}^{I\times J}$, if $I'\subseteq I$ and $J'\subseteq J$, then $A_I$ is the restriction of $A$ to the rows indexed by $I'$ and $A^J$ is the restriction of $A$ to the columns indexed by $J'$.

\subsection{Acknowledgements}

Thank you to Louis Theran, under whose supervision this paper was written, particularly for help with the algebraic geometry in Section \ref{sec:gdvr} and the Vertex Removal Lemma in Section \ref{sec:vrl}.
Thank you to Daniel Bernstein, William Sims and others involved in studying the combinatorial characterisation of $\mathcal{M}(\mathrm{Gr}(2,N))$.
Thank you as well to Peiran Wu for helpful discussions.

\section{Generic $d$-Volume Rigidity}\label{sec:gdvr}

Let $\Sigma$ be a $d$-dimensional simplicial complex on $n$ vertices labelled $\{1,\dots,n\}$.

A \defn{framework} of $\Sigma$ in $\mathbb{R}^d$ is a pair $(\Sigma,p)$, where $p=(p(1),\dots,p(n))\in(\mathbb{R}^d)^n$ is a \defn{configuration} vector.
We may measure the (signed) $d$-volumes of different frameworks of $\Sigma$ via the polynomial map
\begin{equation*}
    \alpha_{\Sigma}:(\mathbb{R}^d)^n\rightarrow\mathbb{R}^{f(\Sigma)_d};p\mapsto\left(\begin{vmatrix}1&\dots&1\\p(i_1)&\dots&p(i_{d+1})\end{vmatrix}:i_1\dots i_{d+1}\in\Sigma_d\right)
\end{equation*}
We do not outright require that $\Sigma$ be pure, however we will only measure pure complexes going forward as all maximal simplices of dimension less than $d$ will not be measured by the $\alpha_{\Sigma}$, and will therefore not make any meaningful contribution.

We say two frameworks $(\Sigma,p)$ and $(\Sigma,q)$ are \defn{($d$-volume) equivalent} if
\begin{equation*}
    \alpha_{\Sigma}(p)=\alpha_{\Sigma}(q).
\end{equation*}
Denote by $K_n^d$ the complete $d$-dimensional simplicial complex on $n$ vertices, and write $\alpha_n^d:=\alpha_{K_n^d}$. Then $(\Sigma,p)$ and $(\Sigma,q)$ are \defn{($d$-volume) congruent} if
\begin{equation*}
    \alpha_n^d(p)=\alpha_n^d(q).
\end{equation*}

Now, a framework $(\Sigma,p)$ is \defn{(locally) ($d$-volume) rigid} in $\mathbb{R}^d$ if there exists an open neighbourhood $U$ of $p$ in $(\mathbb{R}^d)^n$ so that, for all $q\in U$, if $(\Sigma,p)$ and $(\Sigma,q)$ are $d$-volume equivalent, then they are $d$-volume congruent.

A framework $(\Sigma,p)$ is \defn{globally ($d$-volume) rigid} in $\mathbb{R}^d$ if, for all $q\in(\mathbb{R}^d)^n$, if $(\Sigma,p)$ and $(\Sigma,q)$ are $d$-volume equivalent, then they are $d$-volume congruent.

A point $p\in X\subseteq\mathbb{R}^D$ is \defn{generic} if, for all $f\in\mathbb{Q}[x_1,\dots,x_{dn}]$, if $f(X)\neq 0$, then $f(p)\neq 0$.
A property $\mathcal{P}:X\rightarrow\{\text{True}\text{, }\text{False}\}$ is \defn{generic} if it is constant over all generic points in $X$.

\begin{proposition}(\cite{southgate2023bounds})
\begin{itemize}
    \item Local $d$-volume rigidity is a generic property of $d$-dimensional simplicial complexes
    \item Global $d$-volume rigidity is not a generic property of $d$-dimensional simplicial complexes.
\end{itemize}
\end{proposition}

Ie. If $p\in(\mathbb{R}^d)^n$ is generic and $(\Sigma,p)$ is locally $d$-volume rigid, then $(\Sigma,q)$ is locally $d$-volume rigid, for all $q\in(\mathbb{R}^d)^n$ generic.
However, there exist $d$-dimensional simplicial complexes $\Theta$ and generic frameworks $(\Theta,p)$ and $(\Theta,q)$ so that $(\Theta,p)$ is globally $d$-volume rigid, but $(\Theta,q)$ is locally, but not globally, $d$-volume rigid (abbreviated to \defn{LNGR}).
Take, for example, the pentagonal bipyramid in $\mathbb{R}^2$.

If all generic frameworks of a simplicial complex $\Sigma$ are globally $d$-volume rigid, then we call $\Sigma$ \defn{generically globally ($d$-volume) rigid} in $\mathbb{R}^d$.
The most obvious family of such simplicial compelexes is the family of complete $d$-dimensional simplicial complexes.

The differential of the complete $d$-volume measurement map $\alpha_n^d$ is the matrix of indeterminates $\mathrm{d}\alpha_n^d$ which, when evaluated at a configuration $p\in(\mathbb{R}^d)^n$, yields the \defn{complete rigidity matrix} $R(p)$.

\begin{lemma}\label{lem:rankofrigmatroid}(\cite{southgate2023bounds})
Let $\Sigma$ be a pure $d$-dimensional simplicial complex on $n$ vertices and let $p\in(\mathbb{R}^d)^n$ be a configuration in $\mathbb{R}^d$. If $n\geq d+1$ and the restriction of the rigidity matrix $R(p)$ to the rows indexed by the $d$-simplices of $\Sigma$, $R(p)_{\Sigma_d}$, has rank $dn-(d^2+d-1)$, then $(\Sigma,p)$ is $d$-volume rigid in $\mathbb{R}^d$.
If $p$ is generic then the reverse direction holds: if $(\Sigma,p)$ is $d$-volume rigid in $\mathbb{R}^d$, then $\rank(R(p)_{\Sigma_d})=dn-(d^2+d-1)$.
\end{lemma}

Lemma \ref{lem:rankofrigmatroid} is proved by observing the space of \defn{infinitesimal flexes} of the framework $(\Sigma,p)$, ie. the infinitesimal velocities of the continuous deformations (flexes) of the vertices.
It is noted, as in \citet{asimow1978rigidity}, that these form the kernel of $R(p)_{\Sigma_d}$, and that the infinitesimal velocities of the trivial flexes account for a $(d^2+d-1)$-dimensional subspace of the space of trivial flexes.
Therefore $dn-(d^2+d-1)$ is the maximum rank that $R(p)$ may achieve and may be achieved by $R(p)_{\Sigma_d}$ if $\Sigma$ has the same number of infinitesimal flexes as the complete graph.

We continue with constructions drawn from the Euclidean graph rigidity setting by using the complete rigidity matrix to define a matroid on the maximal simplices of $K_n^d$, as in \citet{graver1993combinatorial}.

A \defn{matroid} $\mathcal{M}$ is a pair $(E,\mathcal{I})$, where $E$ is a finite \defn{ground set}, and $\mathcal{I}\subseteq\mathcal{P}(E)$ is the set of subsets of $E$ that are \defn{independent} in $\mathcal{M}$.
If $D\subseteq E$, but $D\not\in\mathcal{I}$, then we say that $D$ is \defn{dependent} in $\mathcal{M}$.
If $B\in\mathcal{I}$ and $|I|\leq|B|$, for all $I\subseteq\mathcal{I}$, then $B$ is a \defn{basis} of $\mathcal{M}$, all bases of $\mathcal{M}$ have the same size, called the \defn{rank} of $\mathcal{M}$.
Denote by $\mathcal{B}$ the set of bases of $\mathcal{M}$

Although we defined a matroid as the pair $(E,\mathcal{I})$, matroids are also uniquely determined in terms of their bases (as well as many other equivalent ways).

Associated to $R(p)$ is the row-matroid $\mathcal{F}_n^d(p)$, called the \defn{infinitesimal rigidity matroid} of $(K_n^d,p)$.
For generic configurations $p$ and $q$, $\mathcal{F}_n^d(p)=\mathcal{F}_n^d(q)=:\mathcal{R}_n^d$, call $\mathcal{R}_n^d$ the \defn{($d$-dimensional) rigidity matroid}.

\begin{proposition}\label{prop:basis}
The $d$-dimensional simplicial complex $\Sigma$ on $n\geq d+1$ vertices is locally $d$-volume rigid in $\mathbb{R}^d$ if and only if the $\Sigma_d$ contains a basis of $\mathcal{R}_n^d$.
\end{proposition}

The proof Proposition \ref{prop:basis} follows from Lemma \ref{lem:rankofrigmatroid}.

An alternate way of defining the $d$-dimensional rigidity matroid is as an algebraic matroid.
Given an affine algebraic variety $X\subseteq\mathbb{C}^E$, the \defn{algebraic matroid} associated to $X$, denoted $\mathcal{M}(X)$, is the pair $(E,\mathcal{I})$, where $S\in\mathcal{I}$ if and only if $\overline{\pi_S(X)}=\mathbb{C}^S$.
A thorough introduction to algebraic matroids, and in particular their application to Euclidean graph rigidity, may be found in \citet{rosen2020algebraic}.

Define the \defn{($d$-dimensional) measurement variety}, denoted $M_n^d$, to be the Zariski-closure of the image of $(\mathbb{R}^d)^n$ under $\alpha_n$, it is an affine real algebraic variety in $\mathbb{R}^{\binom{[n]}{d+1}}$ defined over $\mathbb{Q}$.

Meanwhile, define the \defn{complex ($d$-dimensional) measurement variety}, denoted $CM_n^d$, to be the Zariski-closure of the image of $(\mathbb{C}^d)^n$ under $\alpha_n$, it is an affine complex algebraic variety in $\mathbb{C}^{\binom{[n]}{d+1}}$.

\begin{theorem}\label{thm:matroids}
The $d$-dimensional rigidity matroid $\mathcal{R}_n^d$ is the algebraic matroid of $CM_n^d$.
\end{theorem}

A \defn{semi-algebraic set} is a set $S\subseteq\mathbb{R}^N$ defined by polynomial equalities and inequalities as follows:
\begin{equation*}
    S=\{p\in\mathbb{R}^N:f_i(p)=0, g_j>0\text{, }\forall i\in I, j\in J\},
\end{equation*}
where $\{f_i,g_j:i\in I,j\in J\}\subseteq\mathbb{R}[x_1,\dots,x_N]$ is finite.
An \defn{algebraic set} $A$ is a semi-algebraic set where $J=\varnothing$, write
\begin{equation*}
    A=\mathbb{V}_{\mathbb{R}}(f_i:i\in I).
\end{equation*}
If $S$ is a semi-algebraic (resp. algebraic) set defined as above with $\{f_i,g_j:i\in I,j\in J\}\subseteq\mathbf{k}[x_1,\dots,x_N]\subseteq\mathbb{R}[x_1,\dots,x_N]$, then we say that $S$ is defined over $\mathbf{k}$.
The \defn{Zariski closure} of a semi-algebraic set $S$, denoted $\overline{S}$, is the smallest algebraic set containing $S$.
If $A\subseteq\mathbb{R}^M$ is an algebraic set and $\psi:A\rightarrow\mathbb{R}^N$ is an algebraic map (ie. $\psi=(\psi_1,\dots,\psi_N)\subseteq(\mathbf{Q}[x_1,\dots,x_M])^N$), then $\psi(A)$ is a semi-algebraic set and $\overline{\psi(A)}$ is an algebraic set.

An  algebraic set $A$ is said to be \defn{irreducible} if there are no algebraic sets $A_1,A_2$ that are proper subsets of $A$ so that $A=A_1\cup A_2$.
A semi-algebraic set is said to be irreducible if its Zariski closure is irreducible as an algebraic set.
The \defn{dimension} of a semi-algebraic (resp. algebraic) set $S\subseteq\mathbb{R}^N$, denoted $\dim(S)$, is the largest $D$ so that there exists an open subset $U\subseteq\mathbb{R}^N$ so that $U\cap S$ is isomorphic to $\mathbb{R}^D$.
A point $p$ in the semi-algebraic (resp. algebraic) set $S$ is smooth if the open neighbourhood of $p$ in $S$ is $\dim(S)$-dimensional.

\begin{lemma}
If $S\subseteq\mathbb{R}^M$ is an irreducible algebraic set defined over $\mathbf{k}$ and $\psi:S\rightarrow\mathbb{R}^N$ is an algebraic map, then $\psi(S)$ is an irreducible semi-algebraic set defined over $\mathbf{k}$.
\end{lemma}

\begin{proof}
Write $T=\overline{\psi(S)}$ and suppose, for the sake of contradiction, that $T_1,T_2$ are algebraic sets that are proper subsets of $T$ with $T=T_1\cup T_2$.
Then
\begin{equation*}
    \begin{split}
    \psi^{-1}(\psi(S))
    &=\psi^{-1}((T_1\cap\psi(S))\cup(T_2\cap\psi(S)))\\
    &=(S\cap\psi^{-1}(T_1\cap\psi(S)))\cup(S\cap\psi^{-1}(T_2\cap\psi(S)))\\
    &=S_1\cup S_2.
\end{split}
\end{equation*}
Each of $S_1$ and $S_2$ are semi-algebraic sets and $\overline{S_1\cup S_2}=S$, indeed
\begin{equation*}
\begin{split}
    S_{\alpha}=\{&p\in\mathbb{R}^M:p\in s\text{, }\psi(p)\in(T_{\alpha}\cap\psi(S))\}\\
    =\{&p\in\mathbb{R}^M:f_i(p)=0\text{, }g_j(p)>0\text{, }h_k(\psi(p))=0\text{, }r_l(\psi(p))>0,\\
    &\forall i\in I_{\alpha}, j\in J_{\alpha}, k\in K_{\alpha}, l\in L_{\alpha}\},
\end{split}
\end{equation*}
where $I_{\alpha}$ and $J_{\alpha}$ are the index sets of the polynomials whose respective equalities and inequalities define $S_{\alpha}$, likewise for $K_{\alpha}$ and $L_{\alpha}$ with $T_{\alpha}$, for $\alpha\in\{1,2\}$.
Since $\psi\in(\mathbb{Q}[x_1,\dots,x_M])^N$, each of the $h_k\circ\psi$ and $r_l\circ\psi$ are polynomials in $\mathbf{k}[x_1,\dots,x_M]$.
Hence $S_1$ and $S_2$ are semi-algebraic sets.

Next, suppose that $p\in S\setminus\overline{S_1\cup S_2}$, then $\psi^{-1}(\psi(p))\in S_1\cup S_2$, a contradiction.
Meanwhile, $S_1\cup S_2\subseteq S$, and $S$ is an algebraic set containing $S_1\cup S_2$, so $\overline{S_1\cup S_2}\subseteq S$.
Therefore $S=\overline{S_1\cup S_2}$.

Finally, since $S=\overline{S_1\cup S_2}=\overline{S_1}\cup\overline{S_2}$, where $\overline{S_1}$ and $\overline{S_2}$ are algebraic sets that are proper subsets of $S$, we have that $S$ is not irreducible, a contradiction.
\end{proof}

Therefore, noting that the complete $d$-volume measurement map $\alpha_n^d$ is algebraic, $\alpha_n^d((\mathbb{R}^d)^n)$ and $M_n^d=\overline{\alpha_n^d((\mathbb{R}^d)^n)}$ are irreducible semi-algebraic and algebraic sets defined over $\mathbb{Q}$ respectively.

\begin{lemma}
Let $S\subseteq\mathbb{R}^M$ and $T\subseteq\mathbb{R}^N$ be irreducible semi-algebraic sets defined over some finite field extension $\mathbf{k}$ of $\mathbb{Q}$, $\psi:S\rightarrow T$ a surjective algebraic map.
If $p\in S$ is generic, then $\psi(p)$ is generic in $T$.
\end{lemma}

\begin{proof}
Let $f\in\mathbb{Q}[y_1,\dots,y_N]$ be a polynomial that does not vanish on $T$ and suppose that $f(\psi(p))=0$.
Then $g=f\circ\psi\in\mathbb{Q}[x_1,\dots,x_M]$ and $g(p)=0$.
Finally, since $f$ does not vanish on all of $T=\psi(S)$, there exists $q\in S$ so that $g(q)=f(\psi(q))\neq 0$.
Therefore $g$ does not vanish on all of $S$, contradicting the genericity of $p$.
\end{proof}

Therefore, if $p$ is a generic configuration in $(\mathbb{R}^d)^n$, then its measurement under $\alpha_n^d$ is a generic point in $M_n^d$, and is therefore smooth.

Let $A\subseteq\mathbb{R}^N$ be a set.
The \defn{complexification} of $A$ is the set
\begin{equation*}
    A^C=A\otimes_{\mathbb{R}}\mathbb{C}=\{a\otimes 1 + b\otimes i:a,b\in A\}(=\{a+bi:a,b\in A\}),
\end{equation*}
along with complex multiplication defined
\begin{equation*}
    (a_1 + b_1i)(a_2+b_2i)=(a_1a_2-b_1b_2)+(a_1b_2+a_2b_1)i,
\end{equation*}
for all $a_1,a_2,b_1,b_2\in A$.
If $A\subseteq\mathbb{R}^M$ and $B\subseteq\mathbb{C}^N$, then the real dimension of $A$ is denoted by $\dim_{\mathbb{R}}(A)$ and the complex dimension of $B$ is denoted by $\dim_{\mathbb{C}}(B)$.

\begin{lemma}\label{lem:complexification}
If $L=\mathbb{V}_{\mathbb{R}}(l_1,\dots,l_t)\subseteq\mathbb{R}^N$ is an algebraic set, where $\deg(l_i)=1$, for $1\leq i\leq t$, then $L^C=\mathbb{V}_{\mathbb{C}}(l_1,\dots,l_t)\subseteq\mathbb{C}^N$.
\end{lemma}

Here, given an ideal $I\subseteq\mathbb{C}[x_1,\dots, x_N]$, we write $\mathbb{V}_{\mathbb{C}}(I)=\{p\in\mathbb{C}^N:f(p)=0\text{, }\forall f\in I\}$.

\begin{proof}
Proceeding directly:
\begin{equation*}
\begin{split}
    \mathbb{V}_{\mathbb{C}}(l_1,\dots,l_t)&=\{(u+vi)\in\mathbb{C}^N:l_j(u+vi)=0\text{, }\forall 1\leq j\leq t\}\\
    &=\{u+vi:u,v\in\mathbb{R}^N:l_j(u+vi)=0\text{, }\forall 1\leq j\leq t\}\\
    &=\{u+vi:u,v\in\mathbb{R}^N:l_j(u)=l_j(v)=0\text{, }\forall 1\leq j\leq t\}\\
    &=\{u+vi:u,v\in\mathbb{R}^N:u,v\in L\}=L^C.
\end{split}
\end{equation*}
\end{proof}

\begin{lemma}\label{lem:compdim}
Suppose that $L=\mathbb{V}_{\mathbb{R}}(l_1,\dots,l_t)\subseteq\mathbb{R}^N$ is an algebraic set, where $\deg(l_1),\dots,\deg(l_t)=1$.
Then $\dim_{\mathbb{R}}(L)=\dim_{\mathbb{C}}(L^C)$.
\end{lemma}

The proof of Lemma \ref{lem:compdim} follows from facts about the intersections of hyperplanes.

\begin{proof}[Proof of Theorem \ref{thm:matroids}]
We will show that, given a generic configuration $p\in(\mathbb{R}^d)^n$ and a subset $H\subseteq\binom{[n]}{d+1}$, the tangent space of the projection of the onto co-ordinates indexed by $H$ of the complex measurement variety, $T_C=T_{(\pi_H\circ\alpha_n^d)(p)}\overline{\pi_H(CM_n^d)}$, is the complexification of the tangent space of the corresponding projection of the measurement variety, $T_R=T_{(\pi_H\circ\alpha_n^d)(p)}\overline{\pi_H(M_n^d)}$.

Firstly, $CM_n^d$ is the image of $(\mathbb{C}^d)^n$ under an algebraic map.
Therefore, if $p$ is generic in $(\mathbb{R}^d)^n$, then $\alpha_n^d(p)$ is generic in $CM_n^d\subseteq\mathbb{C}^{\binom{[n]}{d+1}}$ and $(\pi_H\circ\alpha_n^d)(p)$ is generic in $\overline{\pi_H(CM_n^d)}\subseteq\mathbb{C}^H$.
Therefore $(\pi_H\circ\alpha_n^d)(p)$ is smooth in $\overline{\pi_H(CM_n^d)}$.

Next, notice that the tangent spaces $T_R$ and $T_C$ are defined by the vanishing of ideals $I_R\subseteq\mathbb{R}[x_1,\dots,x_{dn}]$ and $I_C\subseteq\mathbb{C}[z_1,\dots,z_{dn}]$ with the same generating sets.
Indeed, $T_R=\mathbb{V}_{\mathbb{R}}(\mathrm{d}(\pi_H\circ\alpha_n^d)_px:h\in H)$ and $T_C=\mathbb{V}_{\mathbb{C}}(\mathrm{d}(\pi_H\circ\alpha_n^d)z:h\in H)_p$, where $\mathrm{d}(\pi_H\circ\alpha_n^d)_p$ is the differential of the algebraic map $(\pi_H\circ\alpha_n^d)$, and $x$ and $z$ represent variables in $(\mathbb{R}^d)^n$ and $(\mathbb{C}^d)^n$ respectively.

By Lemma \ref{lem:complexification}, $T_C=(T_R)^C$ and so by Lemma \ref{lem:compdim}, \begin{equation*}
\begin{split}
    \dim_{\mathbb{C}}(\overline{\pi_H(CM_n^d)})&=\dim_{\mathbb{C}}(T_C)\\
    &=\dim_{\mathbb{R}}(T_R)\\
    &=\rank(\mathrm{d}(\pi_H\circ\alpha_n^d)_p)\\
    &=\rank(R(p)_H).
\end{split}
\end{equation*}
We have that $\dim_{\mathbb{C}}(\overline{\pi_H(CM_n^d)})$ if and only if $\overline{\pi_H(CM_n^d)}=\mathbb{C}^H$.

Therefore, independent sets in $\mathcal{R}_n^d$, $H\subseteq\binom{[n]}{d+1}$ so that $\rank(R(p)_H)=|H|$, are precisely the independent sets in $\mathcal{M}(CM_n^d)$, $H'\subseteq\binom{[n]}{d+1}$ so that $\overline{\pi_{H'}(CM_n^d)}=\mathbb{C}^{H'}$.
\end{proof}

\begin{example}\label{ex:rigid}
Consider the 2-dimensional simplicial complex $\Sigma$ defined by its maximal simplices
\begin{equation*}
    \Sigma_2=\{123,124,134\},
\end{equation*}
Let $p=(p(1)_1,p(1)_2,\dots,p(4)_2)\in(\mathbb{R}^2)^4$ be a configuration of $\Sigma$ in $\mathbb{R}^2$.

First we calculate the rigidity matrix of $(\Sigma,p)$:
\begin{equation*}
\begin{split}
    R(p)_{\Sigma_2}=&\left[\begin{matrix}
    p(3)_2-p(2)_2 & p(2)_1-p(3)_1 & p(1)_2-p(3)_2 & p(3)_1-p(1)_1 & \\
    p(4)_2-p(2)_2 & p(2)_1-p(4)_1 & p(1)_2-p(4)_2 & p(4)_1-p(1)_1 & \dots \\
    p(4)_2-p(3)_2 & p(3)_1-p(4)_1 & 0 & 0 & 
    \end{matrix}\right. \\
    &\left.\begin{matrix}
    & p(2)_2-p(1)_2 & p(1)_1-p(2)_1 & 0 & 0 \\
    \dots & 0 & 0 & p(2)_2-p(1)_1 & p(1)_1-p(2)_1 \\
    & p(1)_2-p(4)_2 & p(4)_1-p(1)_1 & p(3)_2-p(1)_1 & p(1)_1-p(3)_1
    \end{matrix}\right],
\end{split}
\end{equation*}
which has rank $3=2.4-(2^2+2-1)$, so $(\Sigma,p)$ is rigid in $\mathbb{R}^2$.
Since $f(\Sigma_2)=3$, $\Sigma_2$ forms a basis in $\mathcal{R}_4^2$.

By Theorem \ref{thm:matroids}, $\overline{\pi_{\Sigma_2}(M_4^2)}=\mathbb{R}^{\Sigma_2}$ and $\overline{\pi_{\Sigma_2}(CM_4^2)}=\mathbb{C}^{\Sigma_2}$.
\end{example}

In order to combinatorially characterise the $d$-dimensional rigidity matroid $\mathcal{R}_n^d$, we will relate it to the algebraic matroid of the $(d,n-1)$-Grassmannian variety, the matroid of which has been combinatorially characterised in dimensions 1 and 2.

The \defn{$(d,n-1)$-Grassmannian (variety)}, denoted $\mathrm{Gr}(d,n-1)$, is the projective algebraic variety that parametrises $d$-dimensional linear subspaces of $\mathbb{C}^{n-1}$. It may be embedded in $\mathbb{CP}^{\binom{n-1}{d}-1}$ via the \defn{Plücker embedding}
\begin{equation*}
    \mathrm{Pl}:\mathrm{Gr}(d,n-1)\hookrightarrow\mathbb{CP}^{\binom{n-1}{d}-1},
\end{equation*}
which lists out (as homogeneous co-ordinates) the $d\times d$ minors of the $d\times(n-1)$ basis matrices of the linear subspaces parametrised by $\mathrm{Gr}(d,n-1)$.
Affine charts for $\mathrm{Gr}(d,n-1)$ are defined as follows:
\begin{equation*}
    G_{i_1\dots i_d}=\{L\in\mathrm{Gr}(d,n-1):\sigma(L)=i_1\dots i_d\},
\end{equation*}
where $\sigma(L)$ is the \defn{shape} of $L$, ie. the ordered $d$-tuple listing the columns in which the first non-zero terms appear in the reduced row-echelon basis matrix of $L$.

Notice that for any generic configuration $p\in(\mathbb{R}^d)^n$, a $d$-volume preserving affine transformation of $\mathbb{R}^d$ may be applied so that the resulting \defn{pinned configuration} $\overline{p}$ has \defn{configuration matrix}
\begin{equation*}
    C(\overline{p})=\begin{bmatrix}1&\dots&1\\ \overline{p}(1)&\dots&\overline{p}(n)\end{bmatrix}=\begin{bmatrix}1&\underline{1}^t&1&\dots&1\\ \underline{0}&I_d&\overline{p}(4)&\dots&\overline{p}(n)\end{bmatrix}\in\mathbb{R}^{(d+1)\times n}.
\end{equation*}
Then the submatrices of $C(\overline{p})$ defined by removing the first row and column, as $p$ varies in the open dense subset of generic points in $(\mathbb{R}^d)^n$ are precisely the basis matrices of real $d$-dimensional linear subspaces of $\mathbb{R}^{n-1}$ of shape $1\dots d$.

This describes the real case of a birational equivalence, which is defined formally as follows:

\begin{theorem}\label{thm:borceastreinu}(\cite{borcea2013realizations})
Let $\mathbb{P}(M_n^d)\subseteq\mathbb{CP}^{\binom{n}{d+1}-1}$ be the projectification of $CM_n^d$. Then $\mathbb{P}(M_n^d)$ is birationally equivalent to $\mathrm{Gr}(d,n-1)$ via the mutually-inverse rational maps
\begin{equation*}
\begin{split}
    \phi&:G_{1\dots d}\rightarrow V=\{\alpha(p):\alpha_n^d(p)_{1\dots (d+1)}\neq0\},\\
    \pi_{\binom{[n]\setminus\{1\}}{d}}&:V\rightarrow G_{1\dots d}.
\end{split}
\end{equation*}
\end{theorem}

The map $\phi$ is the linear map described pre-composing the $d$-coboundary map on the complete $d$-dimensional simplicial complex, $\delta^d$, represented by the $\binom{[n]}{d+1}\times\binom{[n]}{d}$ matrix of $0$s, $1$s and $-1$s, $D^d$, with the projection onto those entries indexed by $(d-1)$-simplices not containing vertex 1.
Represent $\phi$ by the matrix $\Phi$, where $\Phi=(D^d)^{\binom{[n]\setminus\{1\}}{d}}$.

\begin{proposition}
The set $S\subseteq\binom{[n]}{d+1}$ is independent in $\mathcal{R}_n^d$ if and only if the set $T\subseteq\binom{[n]\setminus\{1\}}{d}$ is independent in $\mathcal{M}(G_{1\dots d})$, where $T$ indexes a set of columns forming a basis of $\colspace(\Phi_S)$. 
\end{proposition}

\begin{proof}
The set $S\subseteq\binom{[n]}{d+1}$ is independent in $\mathcal{R}_n^d$ if and only if $\overline{\pi_S(V)}=\mathbb{C}^S$.
This projection is equivalent to the image of $G_{1\dots d}$ under $\Phi_S$ being full-dimensional, ie. $\Phi_S^T$ being full-rank and $T$ being independent in $\mathcal{M}(G_{1\dots d})$.
\end{proof}

\begin{corollary}\label{cor:acyclic}
If the $d$-dimensional simplicial complex $\Sigma$ is independent in $\mathcal{R}_n^d$, then $\Sigma$ is a-$d$-cyclic (ie. $\beta(\Sigma)_d=0$).
\end{corollary}

\begin{proof}
Suppose that $\beta(\Sigma)_d>0$ and fix a $d$-cycle, supported by the maximal simplices of the $d$-dimensional simplicial sub-complex $\Gamma$ of $\Sigma$.

Since $\Phi_{\Sigma_d}$ is the restriction of the matrix $D^d(\Sigma)$ to the columns indexed by $\Sigma_{d-1}\setminus\link(1)$, if $\Gamma$ is a $d$-cycle, then it induces a row-dependency on $D^d(\Sigma)$ and hence on $D^d(\Sigma)^{\Sigma_{d-1}\setminus\link(1)}=\Phi_{\Sigma_d}$.

If $\Phi$ admits a row-dependency, then $\rank(\Phi)<f(\Sigma)_d$ and hence the rank of the corresponding set in $\mathcal{M}(\mathrm{Gr}(d,n-1))$ will be deficient and so it will be dependent in $\mathcal{M}(\mathrm{Gr}(d,n-1))$, therefore $\Sigma$ will be dependent in $\mathcal{R}_n^d$.
\end{proof}

By this characterisation of $\mathcal{R}_n^d$, we may characterise $d$-volume rigidity when $d\in\{1,2\}$ in terms of the combinatorics of simplicial complexes as follows:

\begin{theorem}\label{thm:combchar}
Let $\Sigma$ be a $d$-dimensional simplicial complex on $n$ vertices,
\begin{itemize}
    \item If $d=1$, $\Sigma$ is rigid if and only if $\Sigma$ is connected.
    \item If $d=2$, $\Sigma$ is rigid if and only if $\Sigma$ admits a spanning 2-dimensional sub-complex $\Sigma'$ so that the indexing set $S$ of a column basis of $\colspace(\Phi_{\Sigma'_2})$ has size $2n-5$ and the graph $([n]\setminus\{1\},S)$ admits an acyclic, ACT-free orientation.
\end{itemize}
\end{theorem}

The $d=1$ case is folkloric and follows from the same argument as the Euclidean graph rigidity case (see \citet{graver1993combinatorial}), since the signed 1-volume of a 1-simplex in $\mathbb{R}^1$ is precisely its signed length.
The $d=2$ case utilises Bernstein's combinatorial characterisation of independence in the algebraic matroid of the $(2,N)$-Grassmannian.

\begin{theorem}(\cite{bernstein2017completion})\label{thm:bernstein}
The set $E\subseteq\binom{[N]}{2}$ is independent in the matroid $\mathcal{M}(\mathrm{Gr}(2,N))$ if and only if $E$ admits an acyclic graph orientation that is ACT-free.
\end{theorem}

Here an acyclic graph orientation corresponds to a (re-)labelling of the vertices of a graph $G=(V,E)$, and an ACT is a cycle of $G$ as a directed graph ($E\subseteq V\times V$), $i_1,i_2,\dots, i_{2k}i_1$, so that $(i_1,i_2),(i_1,i_{2k}),(i_3,i_2),(i_3,i_4),\dots, (i_{2k-1},i_{2k})\in E$.

The test for independence given in \citet{bernstein2017completion} is known to be in RP. Attempts so far to find an algorithm in P have been unsuccessful.
However, implicit in the combinatorial characterisation defined here is another RP test for independence in $\mathcal{M}(\mathrm{Gr}(2,n-1))$: adjoin vertex 1 to each element of $S\subseteq\binom{[n]\setminus\{1\}}{2}$ to get $\Sigma_2$, then $S$ is independent if the set of rows of $R(p)$ indexed by $\Sigma_2$, where $p\in(\mathbb{R}^2)^n$ is a randomly chosen vector, is linearly independent.

In the proof of Theorem \ref{thm:combchar}, we determine what is required to be a basis in $\mathcal{R}_n^d$, noting that all rigid simplicial complexes' vertices are spanned by the $d$-simplices contained in some basis.

\begin{proof}[Proof of Theorem \ref{thm:combchar}]
\begin{itemize}
    \item Let $d=1$.
    Then $\Sigma_1\subseteq\binom{[n]}{2}$ is a basis in $\mathcal{R}_n^1$ if and only if $|\Sigma_1|=n-1$.
    By the pureness of $\Sigma$, this is equivalent to $\Sigma_1$ being a tree  graph (or a connected and a-1-cyclic 1-dimensional simplicial complex).
    \item Let $d=2$.
    By definition, $\Sigma$ is rigid if and only if $\Sigma_2$ admits a spanning basis of $\mathcal{R}_n^2$, $\Sigma'_2$.
    By Theorem \ref{thm:borceastreinu}, $\Sigma'_2$ corresponds to the basis $S$ of $\mathcal{M}(\mathrm{Gr}(2,n-1))$, an independent set in $\mathcal{M}(\mathrm{Gr}(2,n-1))$ of size $2n-5$.
    By Theorem \ref{thm:bernstein} $S$ is independent if and only if, $([n]\setminus\{1\},S)$ admits an acyclic, ACT-free orientation.
\end{itemize}
\end{proof}

We end with a couple of examples:

\begin{example}\label{ex:matroidindep}
Let $n=5$ and $d=2$.
Then $G_{12}$ is the affine chart of $\mathrm{Gr}(2,4)$ parametrised by $2\times 4$ matrices of the form
\begin{equation*}
    C=\begin{bmatrix}1&0&a&b\\0&1&c&d\end{bmatrix}\in\mathbb{C}^{2\times4}.
\end{equation*}
The rational map $\phi:G_{12}\rightarrow V\subseteq CM_5^2$ is described by the matrix
\begin{equation*}
    \Phi=\begin{bmatrix}
    1&0&0&0&0&0\\
    0&1&0&0&0&0\\
    0&0&1&0&0&0\\
    0&0&0&1&0&0\\
    0&0&0&0&1&0\\
    0&0&0&0&0&1\\
    1&-1&0&1&0&0\\
    1&0&-1&0&1&0\\
    0&1&-1&0&0&1\\
    0&0&0&1&-1&1
    \end{bmatrix}
\end{equation*}
and, as a set,
\begin{equation*}
\begin{split}
    \Phi(G_{12})=\left\{\alpha_5^2\left(p)\right):C(\overline{p})=C\right\},
\end{split}
\end{equation*}
where $C$ varies over matrices as above.

Consider $\Sigma$ defined by its maximal simplices
\begin{equation*}
    \Sigma_2=\{123,124,125,134,135\}.
\end{equation*}
Then $\Phi_{\Sigma_2}$ has a column basis indexed by $S=\{23,24,25,34,35\}$.
The graph $([5]\setminus\{1\},S)$ admits an ACT-free orientation, as in Figure \ref{fig:matroidindep}.
Hence $S$ is independent in $\mathcal{M}(\mathrm{Gr}(2,4))$, and so $\Sigma_2$ is independent in $\mathcal{M}(CM_5^2)$.
Since, moreover, $|S|=5=2(5)-5$, so $S$ is a basis of $\mathcal{M}(\mathrm{Gr}(2,4))$.
Therefore, by Theorem \ref{thm:combchar}, $\Sigma$ is rigid.
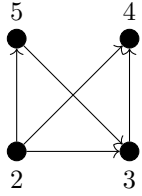
\begin{figure}\label{fig:matroidindep}
    \centering
    \begin{tikzpicture}
        \coordinate[Bullet=black, label=below:2] (n1) at (0,0);
            \coordinate[Bullet=black, label=below:3] (n2) at (1.5,0);
            \coordinate[Bullet=black, label=above:4] (n3) at (1.5,1.5);
            \coordinate[Bullet=black, label=above:5] (n4) at (0,1.5);
            \draw[->] (n1) to (n2);
            \draw[->] (n1) to (n3);
            \draw[->] (n1) to (n4);
            \draw[->] (n2) to (n3);
             \draw[->] (n4) to (n2);
    \end{tikzpicture}
    \caption{An acyclic, ACT-free orientation of $([5]\setminus\{1\},S)$.}
\end{figure}
\end{example}

\begin{example}\label{ex:matroiddep}
Consider the triangular bipyramid $\Sigma$, defined
\begin{equation*}
    \Sigma_2=\{123,124,134,235,245,345\}.
\end{equation*}
Then
\begin{equation*}
    \Phi_{\Sigma_2}=\begin{bmatrix}
        1&0&0&0&0&0\\
        0&1&0&0&0&0\\
        0&0&0&1&0&0\\
        1&0&-1&0&1&0\\
        0&1&-1&0&0&1\\
        0&0&0&1&-1&1
    \end{bmatrix},
\end{equation*}
but $\rank(\Phi_{\Sigma_2})=5<|\Sigma_2|$, therefore, $\Sigma_2$ will be dependent in $\mathcal{R}_5^2$.

Indeed, $\Sigma$ is a 2-cycle, so we have illustrated Corollary \ref{cor:acyclic}.
\end{example}

\section{Exterior Algebraic Shifting}\label{sec:eas}

A (not necessarily pure) $d$-dimensional simplicial complex $\Sigma$ on vertex set $[n]$ is \defn{shifted} with respect to some linear extension $<_{\ell}$ of the lexicographically induced partial order on $\binom{[n]}{d+1}$ if, for each $0\leq k\leq d$, if $i_1\dots i_{k+1}\in\Sigma_k$, then $j_1\dots j_{k+1}\in\Sigma_k$, for all $j_1\dots j_{k+1}<_{\ell} i_1\dots i_{k+1}$.

Given a linear extension $<_{\ell}$ of the partial ordering of $\binom{[n]}{d+1}$, exterior algebraic shifting is an algorithm ascribing to each $d$-dimensional simplicial complex $\Sigma$, its \defn{exterior algebraically shifted complex} (with respect to $<_{\ell}$) $\Delta(\Sigma):=\Delta^{\text{ext}}_{<_{\ell}}(\Sigma)$.
The steps of the algorithm as outlined in \citet{kalai2002algebraic} are as follows:
\begin{enumerate}
    \item Define an $n\times n$ matrix $X=(x_{i,j})_{i,j\in[n]}$ of indeterminates $x_{1,1},\dots,x_{n,n}$;
    \item For each $0\leq k\leq d$:
    \begin{enumerate}
        \item define the $k^{th}$ compound matrix
        \begin{equation*}
            X^{\wedge k}=\left(\begin{vmatrix}x_{i_1,j_1}&\dots& x{i_1,j_{k+1}}\\ \vdots&\ddots&\vdots\\ x_{i_{k+1},j_1}&\dots& x_{i_{k+1},j_{k+1}}\end{vmatrix}\right)_{i_1\dots i_{k+1},j_1\dots j_{k+1}\in\binom{[n]}{k+1}}
        \end{equation*}
        and let $X^{\wedge k}_{\Sigma_k}$ be the restriction of $X^{\wedge k}$ to the rows indexed by $\Sigma_k$;
        \item Choose a greedy basis with respect to $<_{\ell}$ from the columns of $X^{\wedge k}$ for $\colspace(X^{\wedge k}_{\Sigma_k})$ and define $\Delta(\Sigma)_k$ to be the set of indices of the columns chosen.
    \end{enumerate}
    \item Define $\Delta(\Sigma)$ to be the $d$-dimensional simplicial complex with $k$ simplices $\Delta(\Sigma)_k$, for all $0\leq k\leq d$.
\end{enumerate}

Exterior algebraic shifting has been applied to $d$-volume rigidity before:

\begin{theorem}\label{thm:bnp}(\cite{bulavka2022volume})
The $d$-dimensional simplicial complex $\Sigma$ is $d$-volume rigid if and only if there exists $<_{\ell}$ as above so that $134\dots(d+1)n\in\Delta^{\text{ext}}_{<_{\ell}}(\Sigma)_d$.
\end{theorem}

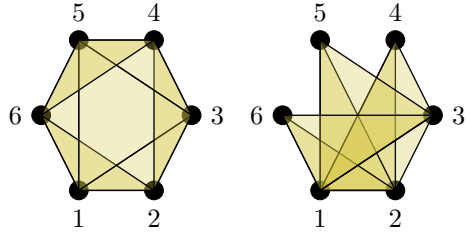
\begin{figure}
    \centering
    \begin{tikzpicture}
        \coordinate[Bullet=black, label=below:1] (n1) at (0,0);
            \coordinate[Bullet=black, label=below:2] (n2) at (1,0);
            \coordinate[Bullet=black, label=right:3] (n3) at (1.5,1);
            \coordinate[Bullet=black, label=above:4] (n4) at (1,2);
            \coordinate[Bullet=black, label=above:5] (n5) at (0,2);
            \coordinate[Bullet=black, label=left:6] (n6) at (-0.5,1);

            \filldraw[fill=yellow!80!black,line width=0.5pt,fill opacity=0.3] (0,0) -- (1,0) -- (1.5,1) -- (0,0);
            \filldraw[fill=yellow!80!black,line width=0.5pt,fill opacity=0.3] (0,0) -- (0,2) -- (-0.5,1) -- (0,0);
            \filldraw[fill=yellow!80!black,line width=0.5pt,fill opacity=0.3] (0,0) -- (-0.5,1) -- (1,0) -- (0,0);
            \filldraw[fill=yellow!80!black,line width=0.5pt,fill opacity=0.3] (0,0) -- (1.5,1) -- (0,2) -- (0,0);
            \filldraw[fill=yellow!80!black,line width=0.5pt,fill opacity=0.3] (1,0) -- (1.5,1) -- (1,2) -- (1,0);
            \filldraw[fill=yellow!80!black,line width=0.5pt,fill opacity=0.3] (1.5,1) -- (1,2) -- (0,2) -- (1.5,1);
            \filldraw[fill=yellow!80!black,line width=0.5pt,fill opacity=0.3] (1,2) -- (0,2) -- (-0.5,1) -- (1,2);
    \end{tikzpicture}
    \begin{tikzpicture}
        \coordinate[Bullet=black, label=below:1] (n1) at (0,0);
            \coordinate[Bullet=black, label=below:2] (n2) at (1,0);
            \coordinate[Bullet=black, label=right:3] (n3) at (1.5,1);
            \coordinate[Bullet=black, label=above:4] (n4) at (1,2);
            \coordinate[Bullet=black, label=above:5] (n5) at (0,2);
            \coordinate[Bullet=black, label=left:6] (n6) at (-0.5,1);

            \filldraw[fill=yellow!80!black,line width=0.5pt,fill opacity=0.3] (0,0) -- (1,0) -- (1.5,1) -- (0,0);
            \filldraw[fill=yellow!80!black,line width=0.5pt,fill opacity=0.3] (0,0) -- (1,0) -- (1,2) -- (0,0);
            \filldraw[fill=yellow!80!black,line width=0.5pt,fill opacity=0.3] (0,0) -- (1,0) -- (0,2) -- (0,0);
            \filldraw[fill=yellow!80!black,line width=0.5pt,fill opacity=0.3] (0,0) -- (1,0) -- (-0.5,1) -- (0,0);
            \filldraw[fill=yellow!80!black,line width=0.5pt,fill opacity=0.3] (0,0) -- (1.5,1) -- (1,2) -- (0,0);
            \filldraw[fill=yellow!80!black,line width=0.5pt,fill opacity=0.3] (0,0) -- (1.5,1) -- (0,2) -- (0,0);
            \filldraw[fill=yellow!80!black,line width=0.5pt,fill opacity=0.3] (0,0) -- (1.5,1) -- (-0.5,1) -- (0,0);
    \end{tikzpicture}
    \caption{The complex on the left is a-2-cyclic and rigid (in fact it is a basis in $\mathcal{R}_6^2$, exterior algebraically shifting it with respect to the complete lexicographic ordering of $\binom{[6]}{2}$ yields the complex on the right.}
    \label{fig:Shift}
\end{figure}

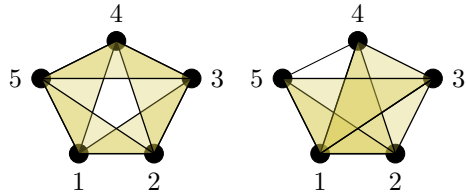
\begin{figure}
    \centering
    \begin{tikzpicture}
        \coordinate[Bullet=black, label=below:1] (n1) at (0,0);
            \coordinate[Bullet=black, label=below:2] (n2) at (1,0);
            \coordinate[Bullet=black, label=right:3] (n3) at (1.5,1);
            \coordinate[Bullet=black, label=above:4] (n4) at (0.5,1.5);
            \coordinate[Bullet=black, label=left:5] (n5) at (-0.5,1);

            \filldraw[fill=yellow!80!black,line width=0.5pt,fill opacity=0.3] (0,0) -- (1,0) -- (1.5,1) -- (0,0);
            \filldraw[fill=yellow!80!black,line width=0.5pt,fill opacity=0.3] (0,0) -- (0.5,1.5) -- (-0.5,1) -- (0,0);
            \filldraw[fill=yellow!80!black,line width=0.5pt,fill opacity=0.3] (0,0) -- (-0.5,1) -- (1,0) -- (0,0);
            \filldraw[fill=yellow!80!black,line width=0.5pt,fill opacity=0.3] (1,0) -- (1.5,1) -- (0.5,1.5) -- (1,0);
            \filldraw[fill=yellow!80!black,line width=0.5pt,fill opacity=0.3] (1.5,1) -- (0.5,1.5) -- (-0.5,1) -- (1.5,1);
    \end{tikzpicture}
    \begin{tikzpicture}
        \coordinate[Bullet=black, label=below:1] (n1) at (0,0);
            \coordinate[Bullet=black, label=below:2] (n2) at (1,0);
            \coordinate[Bullet=black, label=right:3] (n3) at (1.5,1);
            \coordinate[Bullet=black, label=above:4] (n4) at (0.5,1.5);
            \coordinate[Bullet=black, label=left:5] (n5) at (-0.5,1);

            \filldraw[fill=yellow!80!black,line width=0.5pt,fill opacity=0.3] (0,0) -- (1,0) -- (1.5,1) -- (0,0);
            \filldraw[fill=yellow!80!black,line width=0.5pt,fill opacity=0.3] (0,0) -- (1,0) -- (0.5,1.5) -- (0,0);
            \filldraw[fill=yellow!80!black,line width=0.5pt,fill opacity=0.3] (0,0) -- (1,0) -- (-0.5,1) -- (0,0);
            \filldraw[fill=yellow!80!black,line width=0.5pt,fill opacity=0.3] (0,0) -- (1.5,1) -- (0.5,1.5) -- (0,0);
            \filldraw[fill=yellow!80!black,line width=0.5pt,fill opacity=0.3] (0,0) -- (1.5,1) -- (-0.5,1) -- (0,0);
            \draw[-] (n4) -- (n5);
    \end{tikzpicture}
    \caption{Exterior algebraically shifting a pure complex does not necessarily yield a pure complex. The complex on the right is 2-dimensional but has 45 as a maximal simplex.}
    \label{fig:Shift5cycle}
\end{figure}

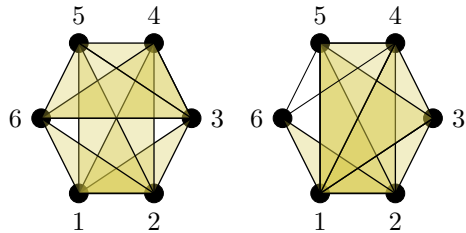
\begin{figure}
    \centering
    \begin{tikzpicture}
        \coordinate[Bullet=black, label=below:1] (n1) at (0,0);
            \coordinate[Bullet=black, label=below:2] (n2) at (1,0);
            \coordinate[Bullet=black, label=right:3] (n3) at (1.5,1);
            \coordinate[Bullet=black, label=above:4] (n4) at (1,2);
            \coordinate[Bullet=black, label=above:5] (n5) at (0,2);
            \coordinate[Bullet=black, label=left:6] (n6) at (-0.5,1);

            \filldraw[fill=yellow!80!black,line width=0.5pt,fill opacity=0.3] (0,0) -- (1,0) -- (1.5,1) -- (0,0);
            \filldraw[fill=yellow!80!black,line width=0.5pt,fill opacity=0.3] (0,0) -- (1,0) -- (1,2) -- (0,0);
            \filldraw[fill=yellow!80!black,line width=0.5pt,fill opacity=0.3] (0,0) -- (1,0) -- (0,2) -- (0,0);
            \filldraw[fill=yellow!80!black,line width=0.5pt,fill opacity=0.3] (0,0) -- (1,0) -- (-0.5,1) -- (0,0);
            \filldraw[fill=yellow!80!black,line width=0.5pt,fill opacity=0.3] (1.5,1) -- (1,2) -- (0,2) -- (1.5,1);
            \filldraw[fill=yellow!80!black,line width=0.5pt,fill opacity=0.3] (1.5,1) -- (1,2) -- (-0.5,1) -- (1.5,1);
            \filldraw[fill=yellow!80!black,line width=0.5pt,fill opacity=0.3] (1.5,1) -- (0,2) -- (-0.5,1) -- (1.5,1);
    \end{tikzpicture}
    \begin{tikzpicture}
        \coordinate[Bullet=black, label=below:1] (n1) at (0,0);
            \coordinate[Bullet=black, label=below:2] (n2) at (1,0);
            \coordinate[Bullet=black, label=right:3] (n3) at (1.5,1);
            \coordinate[Bullet=black, label=above:4] (n4) at (1,2);
            \coordinate[Bullet=black, label=above:5] (n5) at (0,2);
            \coordinate[Bullet=black, label=left:6] (n6) at (-0.5,1);

            \filldraw[fill=yellow!80!black,line width=0.5pt,fill opacity=0.3] (0,0) -- (1,0) -- (1.5,1) -- (0,0);
            \filldraw[fill=yellow!80!black,line width=0.5pt,fill opacity=0.3] (0,0) -- (1,0) -- (1,2) -- (0,0);
            \filldraw[fill=yellow!80!black,line width=0.5pt,fill opacity=0.3] (0,0) -- (1,0) -- (0,2) -- (0,0);
            \filldraw[fill=yellow!80!black,line width=0.5pt,fill opacity=0.3] (0,0) -- (1,0) -- (-0.5,1) -- (0,0);
            \filldraw[fill=yellow!80!black,line width=0.5pt,fill opacity=0.3] (0,0) -- (1.5,1) -- (1,2) -- (0,0);
            \filldraw[fill=yellow!80!black,line width=0.5pt,fill opacity=0.3] (0,0) -- (1.5,1) -- (0,2) -- (0,0);
            \filldraw[fill=yellow!80!black,line width=0.5pt,fill opacity=0.3] (0,0) -- (1,2) -- (0,2) -- (0,0);
            \draw[-] (1,2) -- (0,2);
            \draw[-] (1,2) -- (-0.5,1);
            \draw[-] (0,2) -- (-0.5,1);
    \end{tikzpicture}
    \caption{An example of exterior algebraically shifting a flexible complex.}
    \label{fig:Shiftflexible}
\end{figure}

The following properties of exterior algebraically shifted complexes will be relevant in our discussion of independent (in $\mathcal{R}_n^d$) $d$-volume rigid simplicial complexes:

\begin{lemma}\label{lem:shiftingproperties}(\cite{kalai2002algebraic})
Let $\Sigma$ be a $d$-dimensional simplicial complex, then
\begin{enumerate}
    \item $f(\Delta^{\text{ext}}_{<_{\ell}}(\Sigma))=f(\Sigma)$;
    \item $\beta(\Delta^{\text{ext}}_{<_{\ell}}(\Sigma))=\beta(\Sigma)$;
    \item $\beta(\Delta^{\text{ext}}_{<_{\ell}}(\Sigma))_d=|\{i_1\dots i_{d+1}\in\Sigma_d:1\not\in i_1\dots i_{d+1}\}|$;
\end{enumerate}
for any $<_{\ell}$ as above.
\end{lemma}

Notably, if $\Sigma$ is independent in $\mathcal{R}_n^d$, then $f(\Delta^{\text{ext}}_{<_{\ell}}(\Sigma))\leq dn-(d^2+d-1)$ (with equality if and only if $\Sigma$ is a basis) and, by corollary \ref{cor:acyclic}, $\beta(\Delta^{\text{ext}}_{<_{\ell}})_d=0$.
Therefore the complex defined by the $d$-simplices of the exterior algebraically shifted complex of an independent $d$-dimensional simplicial complex is a star with apex vertex 1.

\begin{theorem}\label{thm:bound}
Suppose that $\Sigma$ is a $d$-dimensional simplicial complex on $n$ vertices and is a basis in $\mathcal{R}_n^d$. Then, for all $0\leq k\leq d$:
\begin{equation*}
    f(\Sigma)_k\geq \binom{d+1}{k+1}+\sum\limits_{l=0}^{d-1}(n-d-1)\binom{d-l}{k-l}.
\end{equation*}
Moreover, this bound is sharp.
\end{theorem}

\begin{proof}
Suppose that $\Sigma$ is a $d$-dimensional simplicial complex as in the statement.
Then
\begin{equation*}
    (f(\Sigma)_0,f(\Sigma)_d)=(n,dn-(d^2+d-1)).
\end{equation*}
By Theorem \ref{thm:bnp}, the (lexicographically induced) partial order $\left(\binom{[n]}{d+1},<_p\right)$ admits a linear extension $<_{\ell}$ so that $134\dots(d+1)n\in(\Delta_{<_{\ell}}^{\text{ext}}(\Sigma))_d$.
We will show that there a certain set of $dn-(d^2+d-1)$ $d$-simplices are less than or equal to $134\dots(d+1)n$ with respect to $<_p$, and therefore will be present in any shifted complex with respect to a linear extension of $<_p$.
We will then show that this set of $d$-simplices has the desired number of lower-dimensional simplices.

Now,
\begin{equation*}
    134\dots(d+1)n\geq_p 134\dots(d+1)(n-1)\geq_p\dots\geq_p 134\dots(d+1)(d+2),
\end{equation*}
this is a chain of length $n-d-1$.
Similarly,
\begin{equation*}
    134\dots(d+1)n\geq_p 124\dots(d+1)n\geq_p\dots\geq_p 124\dots(d+1)(d+2),
\end{equation*}
another chain (excluding $134\dots(d+1)n$) of length $n-d-1$.
Indeed, for each $(d-1)$-tuple $i_1\dots i_{d-1}\in\binom{[d+1]\setminus\{1\}}{d-1}$, of which there are $d$, there is a chain of length $n-d-1$ (excluding the first term) of the form
\begin{equation*}
    134\dots(d+1)n\geq_p 1i_1\dots i_{d-1}n\geq_p\dots\geq_p 1i_1\dots i_{d-1}(d+2).
\end{equation*}
If $j_1\dots j_{d-1}\not\in\binom{[d+1]\setminus\{1\}}{d-1}$, and $1<j_1$, then $j_{d-1}>d+1$, and so $1j_1\dots j_{d-1}k\not\leq_p 134\dots(d+1)n$, for any $j_{d-1}<k\leq n$.

This accounts for $d(n-d-1)$ $d$-simplices, there is just one more, the smallest $d$-simplex in any linear extension, $123\dots d(d+1)$.
Therefore there are $d(n-d-1)+1=dn-(d^2+d-1)$ $d$-simplices smaller that $134\dots(d+1)n$ in $\left(\binom{[n]}{d+1},\leq_p\right)$.
Call the set of such simplices $\Lambda_d$.

It remains to show that the sets of $k$-simplices in the simplicial complex $\Lambda$, defined by $\Lambda_d$, have the desired size.

Consider the \textit{stacked Hässe diagram} of $\left(\binom{[n]}{d+1},<_p\right)$ as in figure \ref{fig:stackedhasse}.
We will count the number of distinct $(k+1)$-tuples appearing \textit{diagonal-wise} along the first two diagonals of the first level, and the first diagonals of subsequent levels until arriving at $134\dots(d+1)n$.
Within the first diagonal of the first layer, there are
\begin{equation}\label{eq:l1d1}
    \binom{d}{k+1}+(n-d)\binom{d}{k}
\end{equation}
distinct $k$-simplices, the first term coming from those in the first $d$ terms, the second coming from those featuring one of $d+1,\dots n$.
Within the second diagonal of the first layer, there are
\begin{equation}\label{eq:l1d2}
    (n-d-1)\binom{d-1}{k-1}
\end{equation}
Within the first diagonal of layer $l$, $2\leq l\leq d-1$, as $134\dots(d+1)n$ is contained in the first diagonal of layer $d-1$, there are
\begin{equation}\label{eq:lld1}
    (n-d-1)\binom{d-l}{k-l}.
\end{equation}
Adding together quantities \ref{eq:l1d1}, \ref{eq:l1d2} and each of \ref{eq:lld1} yields the formula
\begin{equation}\label{eq:minksimps}
    f(\Lambda)_k=\binom{d+1}{k+1} +\sum\limits_{l=0}^{d-1}(n-d-1)\binom{d-l}{k-l}.
\end{equation}

Now, suppose that $\Sigma$ is a basis in $\mathcal{R}_n^d$, then $\Delta(\Sigma)=\Delta_{<_{\ell}}^{\text{ext}}(\Sigma)$ contains $134\dots(d+1)n$ as a $d$-simplex.
As $\Delta(\Sigma)$ is a shifted complex, it must at least contain all the $d$-simplices less than $134\dots(d+1)n$ in the partial order.
Therefore $\Lambda_d\subseteq\Delta(\Sigma)_d$ (in fact, since they have equal cardinality, they are equal), and thus $f(\Delta(\Sigma))_k\geq f(\Lambda)_k$, for all $0\leq k\leq d$.
Finally, by Lemma \ref{lem:shiftingproperties},
\begin{equation*}
    f(\Sigma)_k=f(\Delta(\Sigma))_k\geq f(\Lambda)_k=\binom{d+1}{k+1}+\sum\limits_{l=0}^{d-1}(n-d-1)\binom{d-l}{k-l},
\end{equation*}
for all $0\leq k\leq d$, proving the Theorem.
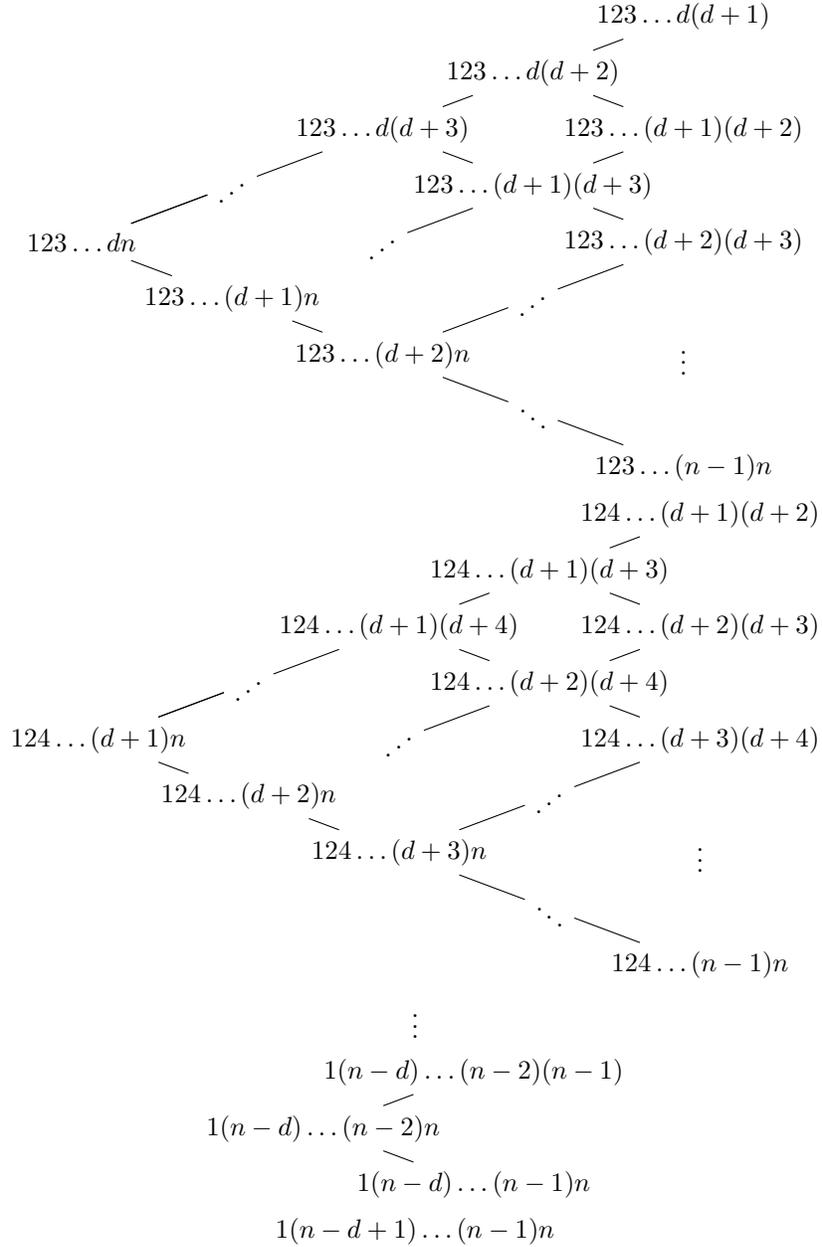
\begin{figure}
    \centering
    \begin{tikzpicture}
		\node (n1) at (0cm,0cm) {$123\dots d(d+1)$};
		\node (n2) at (-2cm,-0.75cm) {$123\dots d(d+2)$};
		\node (n3) at (-4cm,-1.5cm) {$123\dots d(d+3)$};
        \node (n4) at (-6cm,-2.25cm) {\scalebox{-1}[1]{$\ddots$}};
        \node (n5) at (-8cm,-3cm) {$123\dots dn$};
        \node (n6) at (0cm,-1.5cm) {$123\dots(d+1)(d+2)$};
        \node (n7) at (-2,-2.25cm) {$123\dots(d+1)(d+3)$};
        \node (n8) at (-4cm,-3cm) {\scalebox{-1}[1]{$\ddots$}};
        \node (n9) at (-6cm,-3.75cm) {$123\dots(d+1)n$};
        \node (n10) at (0cm,-3cm) {$123\dots(d+2)(d+3)$};
        \node (n11) at (-2cm,-3.75cm) {\scalebox{-1}[1]{$\ddots$}};
        \node (n12) at (-4cm,-4.5cm) {$123\dots(d+2)n$};
        \node (n13) at (0cm,-4.5cm) {$\vdots$};
        \node (n14) at (-2cm,-5.25cm) {$\ddots$};
        \node (n15) at (0cm,-6cm) {$123\dots(n-1)n$};
		\draw[-] (n1) -- (n2);
        \draw[-] (n2) -- (n3);
		\draw[-] (n3) -- (n4);
        \draw[-] (n4) -- (n5);
        \draw[-] (n4) -- (n5);
        \draw[-] (n2) -- (n6);
        \draw[-] (n3) -- (n7);
        \draw[-] (n6) -- (n7);
        \draw[-] (n7) -- (n8);
        \draw[-] (n7) -- (n10);
        \draw[-] (n5) -- (n9);
        \draw[-] (n9) -- (n12);
        \draw[-] (n12) -- (n14);
        \draw[-] (n14) -- (n15);
        \draw[-] (n10) -- (n11);
        \draw[-] (n11) -- (n12);
    \end{tikzpicture}\\
    \begin{tikzpicture}
		\node (n1) at (0cm,0cm) {$124\dots(d+1)(d+2)$};
		\node (n2) at (-2cm,-0.75cm) {$124\dots(d+1)(d+3)$};
		\node (n3) at (-4cm,-1.5cm) {$124\dots(d+1)(d+4)$};
        \node (n4) at (-6cm,-2.25cm) {\scalebox{-1}[1]{$\ddots$}};
        \node (n5) at (-8cm,-3cm) {$124\dots(d+1)n$};
        \node (n6) at (0cm,-1.5cm) {$124\dots(d+2)(d+3)$};
        \node (n7) at (-2,-2.25cm) {$124\dots(d+2)(d+4)$};
        \node (n8) at (-4cm,-3cm) {\scalebox{-1}[1]{$\ddots$}};
        \node (n9) at (-6cm,-3.75cm) {$124\dots(d+2)n$};
        \node (n10) at (0cm,-3cm) {$124\dots(d+3)(d+4)$};
        \node (n11) at (-2cm,-3.75cm) {\scalebox{-1}[1]{$\ddots$}};
        \node (n12) at (-4cm,-4.5cm) {$124\dots(d+3)n$};
        \node (n13) at (0cm,-4.5cm) {$\vdots$};
        \node (n14) at (-2cm,-5.25cm) {$\ddots$};
        \node (n15) at (0cm,-6cm) {$124\dots(n-1)n$};
		\draw[-] (n1) -- (n2);
        \draw[-] (n2) -- (n3);
		\draw[-] (n3) -- (n4);
        \draw[-] (n4) -- (n5);
        \draw[-] (n4) -- (n5);
        \draw[-] (n2) -- (n6);
        \draw[-] (n3) -- (n7);
        \draw[-] (n6) -- (n7);
        \draw[-] (n7) -- (n8);
        \draw[-] (n7) -- (n10);
        \draw[-] (n5) -- (n9);
        \draw[-] (n9) -- (n12);
        \draw[-] (n12) -- (n14);
        \draw[-] (n14) -- (n15);
        \draw[-] (n10) -- (n11);
        \draw[-] (n11) -- (n12);
    \end{tikzpicture}\\
    \begin{tikzpicture}
        \node (n1) at (0cm,0cm) {$\vdots$};
    \end{tikzpicture}\\
    \begin{tikzpicture}
        \node (n1) at (0cm,0cm) {$1(n-d)\dots(n-2)(n-1)$};
        \node (n2) at (-2cm,-0.75cm) {$1(n-d)\dots(n-2)n$};
        \node (n3) at (0cm,-1.5cm) {$1(n-d)\dots(n-1)n$};
        \draw[-] (n1) -- (n2);
        \draw[-] (n2) -- (n3);
    \end{tikzpicture}\\
    \begin{tikzpicture}
        \node (n1) at (0cm,0cm) {$1(n-d+1)\dots(n-1)n$};
    \end{tikzpicture}
    \caption{The stacked Hässe diagram of the partial order on $\binom{[n]}{d+1}$, there are lines between the layers, but they are not drawn.}
    \label{fig:stackedhasse}
\end{figure}
\end{proof}

Call the simplicial complex $\Lambda$ above the \defn{lexicographically greedy rigid complex} (\defn{LGRC}), since it is the greedy (with respect to the lexicographic order) choice of $d$-simplices that index a basis of the row space of the complete rigidity matrix $R(p)$, for some generic $p\in(\mathbb{R}^d)^n$.

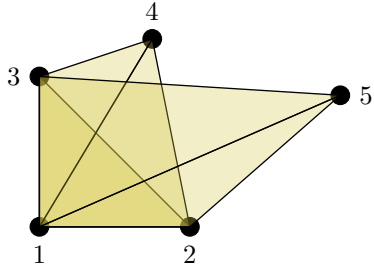
\begin{figure}
    \centering
    \begin{tikzpicture}
        \coordinate[Bullet=black, label=below:1] (n1) at (0,0);
            \coordinate[Bullet=black, label=below:2] (n2) at (2,0);
            \coordinate[Bullet=black, label=left:3] (n3) at (0,2);
            \coordinate[Bullet=black, label=above:4] (n4) at (1.5,2.5);
            \coordinate[Bullet=black, label=right:5] (n5) at (4,1.75);

            \filldraw[fill=yellow!80!black,line width=0.5pt,fill opacity=0.3] (0,0) -- (2,0) -- (0,2) -- (0,0);
            \filldraw[fill=yellow!80!black,line width=0.5pt,fill opacity=0.3] (0,0) -- (2,0) -- (1.5,2.5) -- (0,0);
            \filldraw[fill=yellow!80!black,line width=0.5pt,fill opacity=0.3] (0,0) -- (2,0) -- (4,1.75) -- (0,0);
            \filldraw[fill=yellow!80!black,line width=0.5pt,fill opacity=0.3] (0,0) -- (0,2) -- (1.5,2.5) -- (0,0);
            \filldraw[fill=yellow!80!black,line width=0.5pt,fill opacity=0.3] (0,0) -- (0,2) -- (4,1.75) -- (0,0);
    \end{tikzpicture}
    \caption{The LGRC in dimension 2 on five vertices.}
    \label{fig:LGRC}
\end{figure}

\begin{example}
The bound in Theorem \ref{thm:bound} is not a sufficient condition for volume rigidity.
Indeed the 2-dimensional simplicial complex $\Sigma$ defined by its maximal simplices
\begin{equation*}
    \Sigma_2=\{123,124,125,128,134,135,138,237,267,256,345\}
\end{equation*}
has $f$-vector $f(\Sigma)=(1,8,18,11)$, which is equal to that of the lexicographically greedy complex on eight vertices.
However $\Sigma$ is flexible in $\mathbb{R}^2$.
\end{example}

\begin{lemma}
The $d$-dimensional LGRC on $n$ vertices is generically globally rigid in $\mathbb{R}^d$, for all $d\geq 1$ and $n\geq d+1$.
\end{lemma}

\begin{proof}
Let $p,q\in(\mathbb{R}^d)^n$ be generic configurations.
Suppose that $(\Lambda,p)$ and $(\Lambda,q)$ are equivalent.
Then
\begin{equation}\label{eq:lexcomppin}
    \begin{vmatrix}
        1&\dots&1\\p(1)&\dots&p(d+1)
    \end{vmatrix}=\begin{vmatrix}
        1&\dots&1\\q(1)&\dots&q(d+1)
    \end{vmatrix},
\end{equation}
and, for all $i_1\dots i_d\in\binom{[d+1]}{d}$ and $d+2\leq j\leq n$,
\begin{equation}\label{eq:lexcomp2}
    \begin{vmatrix}
        1&\dots&1&1\\p(i_1)&\dots&p(i_d)&p(j)
    \end{vmatrix}=\begin{vmatrix}
        1&\dots&1&1\\q(i_1)&\dots&q(i_d)&q(j)
    \end{vmatrix}.
\end{equation}
Equation \ref{eq:lexcomppin} uniquely defines the positions of $q(1),\dots, q(d+1)$ as the image of $p(1),\dots, p(d+1)$ (respectively) under a $d$-volume preserving affine transformation of $\mathbb{R}^d$.
Then, fixing $q(1),\dots, q(d+1)$, for each $d+2\leq j\leq n$, the equations of form \ref{eq:lexcomp2} describe $d+1$ linear equations solved by $q(j)$.
By genericity, $q(j)$ is the unique solution to these equations.
Hence, up to a $d$-dimensional affine transformation of $\mathbb{R}^d$, $(\Lambda,q)$ is the unique equivalent framework to $(\Lambda,p)$.
Moreover, by fixing $q(1)=p(1),\dots, q(d+1)=p(d+1)$, we get $q(j)=p(j)$, for all $d+2\leq j\leq n$, so $(\Lambda,q)$ is congruent to $(\Lambda,p)$.
\end{proof}

The LGRC may be interpreted as a $d$-volume rigidity analogue to a trilateration graph.

Such a family of globally rigid bases of rigidity matroids does not exist in the case of Euclidean graph rigidity by Hendrickson's conditions (specifically redundant rigidity being a necessary condition for global rigidity).
The LGRC would fail to be globally rigid if we were to instead measure the absolute or squared $d$-volumes, as reflections of the vertices in the hyperplanes \textit{opposite} them would be allowed.

So in the case of signed $d$-volume rigidity, a complex is rigid if and only if its exterior algebraically shifted complex is rigid, and if a complex is rigid, its exterior algebraically shifted complex is globally rigid.
In the case of unsigned $d$-volume rigidity (in either of the formulations mentioned above), a complex is rigid if and only if its exterior algebraically shifted complex is rigid, however we can not make any statements on the global rigidity of the shifted complex as of yet.

\section{Vertex Removal Lemmas}\label{sec:vrl}

Here we outline an instance where Theorem \ref{thm:bound} is used to give an improved statement of a theorem, that being the volume rigidity analogue of Tanigawa's Vertex Removal Lemma for Euclidean graph rigidity.

Let $G=(V,E)$ be a graph (ie. a 1-dimensional simplicial complex), $p\in(\mathbb{R}^d)^n$ a configuration and $(G,p)$ a framework.
Two frameworks $(G,p)$ and $(G,q)$ are \defn{Euclidean equivalent} if
\begin{equation*}
    \|p(i)-p(j)\|=\|q(i)-q(j)\|,
\end{equation*}
for all $ij\in E$, and \defn{Euclidean congruent} if
\begin{equation*}
    \|p(i)-p(j)\|=\|q(i)-q(j)\|,
\end{equation*}
for all $ij\in\binom{V}{2}$.
The definitions of \defn{(local)} and \defn{global Euclidean rigidity} in $\mathbb{R}^d$ are analogous to those of local and global $d$-volume rigidity.
Unlike in the case of $d$-volume global rigidity, the global Euclidean rigidity in $\mathbb{R}^d$, for any $d$, is a generic property.

\begin{theorem}\cite{tanigawa2015sufficient})
Let $G=(V,E)$ be a graph and let $i\in V$ be a vertex of degree $\deg(i)\geq d+1$. Suppose that
\begin{enumerate}
    \item $G-i$ is rigid in $\mathbb{R}^d$;
    \item $G-i+K(N(i))$ is globally rigid in $\mathbb{R}^d$;
\end{enumerate}
then $G$ is globally rigid in $\mathbb{R}^d$.
\end{theorem}

In the above statement, $N(i)=\{j\in V:ij\in E\}$ denotes the neighbourhood of $i$ and, for any finite set $U$, $K(U):=\left(U,\binom{U}{2}\right)$ is the complete graph on vertex set $U$.

We now state the weaker analogous theorem for $d$-volume rigidity.

\begin{theorem}\label{thm:VRL}
Let $\Sigma$ be a $d$-dimensional simplicial complex and let $i\in\Sigma_0$ be a vertex so that $|\link(i)|\geq d+1$.
Suppose that $\Sigma-i+K(\link(i))$ is generically globally $d$-volume rigid in $\mathbb{R}^d$.
Then $\Sigma$ is generically globally $d$-volume rigid in $\mathbb{R}^d$.
\end{theorem}

Here, for any collection of $k$-simplices, with $0\leq k\leq d$, $U$, $K(U)$ is the smallest complete $d$-dimensional simplicial complex with $K(U)_k\subseteq U$.

\begin{proof}
Suppose that $i<j$, for all $j\in\link(i)_0$.
Let $p\in(\mathbb{R}^d)^n$ be a generic configuration.
Define a $d$-simplex $j_1\dots j_{d+1}$, every $(d-1)$-simplex of which is contained in $\link(i)$.
Then, suppose that $q\in(\mathbb{R}^d)^n$ is a generic configuration and that $(\Sigma+j_1\dots j_{d+1},p)$ and $(\Sigma+j_1\dots j_{d+1},q)$ are equivalent, ie.
$(\Sigma,p)$ and $(\Sigma,q)$ are equivalent and
\begin{equation*}
\begin{split}
    &\begin{vmatrix}
        1&\dots&1\\p(j_1)&\dots&p(j_{d+1})
    \end{vmatrix}-\begin{vmatrix}
        1&\dots&1\\q(j_1)&\dots&q(j_{d+1})
    \end{vmatrix}\\
    =&\sum\limits_{l=1}^{d+1}\begin{vmatrix}
        1&\dots&1&1&1&\dots&1\\p(j_1)&\dots&p(j_{l-1})&p(i)&p(j_{l+1})&\dots&p(j_{d+1})
    \end{vmatrix}\\
    &-\sum\limits_{l=1}^{d+1}\begin{vmatrix}
        1&\dots&1&1&1&\dots&1\\p(j_1)&\dots&p(j_{l-1})&p(i)&p(j_{l+1})&\dots&p(j_{d+1})
    \end{vmatrix}=0.
\end{split}
\end{equation*}
The first equality coming from the fact that the $d+2$ vertices $i,j_1,\dots,j_{d+1}$ make up a complete $d$-dimensional simplicial complex as a sub-complex of $\Sigma+j_1\dots j_{d+1}$, and therefore, any one of its $d$-simplices is completely determined by the signed sum of the others.
The second equality comes from the fact that each $d$-simplex $ij_1\dots j_{l-1}j_{l+1}\dots j_{d+1}$ is, by definition, in $\Sigma$, and so its volumes over $p$ and $q$ are equal.
Hence $\Sigma$ is generically globally rigid if and only if $\Sigma+j_1\dots j_{d+1}$ is generically globally rigid.

Repeating this step progressively for each $j_1\dots j_{d+1}\in K(\link(i))_d$, we get that $\Sigma$ is generically globally rigid if and only if $\Sigma+K(\link(i))$ is generically globally rigid.

Finally, since $i$ is contained in (at least) $d+1$ $d$-simplices in $\Sigma$, if $\Sigma-i+K(\link(i))$ is generically globally rigid, then $\Sigma+K(\link(i))$ is generically globally rigid.
\end{proof}

Note that there is nothing particularly special about the choosing to attach the complete $d$-dimensional simplicial complex to $\Sigma$, in fact we may attach any generically globally rigid complex on at least $d+1$ vertices to obtain the same result.
Therefore, we consider our lexicographically greedy complex from earlier, which has $dn-(d^2+d-1)$ $d$-simplices.

\begin{corollary}\label{cor:VRLCor}
Let $\Sigma$ be a $d$-dimensional simplicial complex and let $i\in\Sigma_0$ be a vertex. Let $\Lambda$ be a copy of the $d$-dimensional LGRC with vertex set the neighbours of $i$ in $\Sigma$.
If $\Lambda_{d-1}\subseteq\link(i)$ and $\Sigma-i+\Lambda$ is generically globally rigid, then $\Sigma$ is generically globally rigid.
\end{corollary}

\begin{proof}
Proceed as in the proof of Theorem \ref{thm:VRL}, except only adding $d$-simplices that are contained in $\Lambda_d$.
\end{proof}

Therefore, the link complex of the vertex we remove need not be as dense as a complete graph, increasing the number of instances where we may apply the Vertex Removal Lemma.

Indeed it was this question that motivated the use of exterior algebraic shifting to find the minimal face numbers required for algebraic shifting.
For a $2$-dimensional simplicial complex $\Sigma$, with $f(\Sigma)_0=n$ and $f(\Sigma)_2=2n-5$, a lower bound of $f(\Sigma)_1\leq 3n-7$ may be achieved reasonably easily using elementary methods.
However, improving the lower bound was beyond the scope of our initial line of attack, and the elementary methods would quickly get out of hand when trying to generalise to higher dimensional complexes.

\bibliographystyle{plainnat}
\bibliography{MinFNos}

\end{document}